\documentclass[A4j]{article}

\usepackage{amsmath}
\usepackage{amssymb}
\usepackage{amscd}
\usepackage{amsthm}
\usepackage{graphicx}

\title{Iterated integrals and the loop product}
\date{}
\author{Koichi Fujii}

\begin{document} 

\newtheorem{theorem}{Theorem}[section]
\newtheorem{lemma}[theorem]{Lemma}
\newtheorem{proposition}[theorem]{Proposition}

\maketitle

\section{Introduction} 

The purpose of this paper is to describe string topology from the viewpoint of Chen's iterated integrals. Let $M$ be a compact closed oriented $d$-manifold and $LM$ be the free loop space of $M$, the set of unbased smooth maps from $S^1$ to $M$. 
Let $\mathbb{H}_*(LM)$  be the homology of the free loop space shifted by the dimension of the manifold i.e. $\mathbb{H}_*(LM)$ = $H_{*+d}(LM)$. Chas and Sullivan found the product on $\mathbb{H}_*(LM)$  which they called {\it loop product} \cite{chassullivan}:
$$
\mathbb{H}_p(LM) \otimes \mathbb{H}_q(LM) \rightarrow \mathbb{H}_{p+q}(LM).
$$
They showed that this product makes $\mathbb{H}_*(LM)$ an associative, commutative algebra.

Merkulov constructed a model for this product based on the theory of iterated integrals, especially of the formal power series connection \cite{merkulov}. He showed that there is an isomorphism of algebras
$$
\mathbb{H}_*(LM) \cong H_*(\Lambda M \otimes \mathbb{R} \bigl\langle \langle X \rangle \bigr\rangle )
$$
where $\Lambda M$ is the de Rham differential graded algebra of $M$ and $\mathbb{R} \bigl\langle \langle X \rangle \bigr\rangle$ is the formal completion of the free graded associative algebra generated by some noncommutative indeterminates.

On the other hand, Chen showed that the cohomology of the free loop space of the simply-connected manifold is isomorphic to the cohomology of the cyclic bar complex of differential forms via Chen's iterated integrals (see \cite{chen77.2} or \cite{gjp}):
$$
H^*(LM) \cong H^*(C(\Lambda M)).
$$

In this paper, we construct a model for the loop product based on the theory of the cyclic bar complex. We define a complex ${\rm{Hom}}(B(\Lambda M), \Lambda M)$ and its subcomplex ${\rm{\overline{Hom}}}(B(\Lambda M), \Lambda M)$ so that the $\rm{Poincar\acute{e}}$ duality induces the isomorphism of vector spaces
$$
H_*({\rm{Hom}}(C(\Lambda M), \mathbb{R})) \cong H_{*-d}(\overline{{\rm{Hom}}}(B(\Lambda M), \Lambda M)).
$$
We can define a product on ${\rm{\overline{Hom}}}(B(\Lambda M), \Lambda M)$ which realizes the loop product.
\begin{theorem}
 Let $M$ be a compact closed oriented simply-connected manifold. Assume that $H_*(M)$ is of finite type. Let $A$ be a differential graded subalgebra of $\Lambda M$ such that $H^*(A)$ $\cong$ $H^*(\Lambda M)$ by the inclusion. Then there is an isomorphism of associative, commutative algebras 
$$
\mathbb{H}_*(LM) \cong H_*(\overline{{\rm{Hom}}}(B(A), A)). 
$$ 
The product defined on $H_*(\overline{{\rm{Hom}}}(B(A), A))$ corresponds to the loop product under the isomorphism.
\end{theorem}

The paper is organized in the following way. In section 2, we briefly review Chen's iterated integrals.  In section 3, we give a construction of  a complex ${{\rm{Hom}}}(B(A), A)$, and discuss its properties. In section 4, we give a proof of theorem 1.1. In section 5, we study the iterated integrals on the free loop space of the non-simply-connected manifolds. In section 6, we describe a relation between the product on ${{\rm{Hom}}}(B(A), A)$ and the Goldman bracket. In this paper, all the homologies have their coefficients in the field of real numbers.

{\it Acknowledgement}: The author would like to thank Professor Toshitake Kohno much for helpful comments and gentle support.

\section{Chen's iterated integrals}
We briefly review Chen's iterated integrals (see \cite{chen77.2}, or \cite{gjp}). Let $M$ be a finite dimensional smooth manifold and let $LM$ be the free loop space of $M$, that is the space of all smooth maps from $S^1$ to $M$. Let $\Delta_k$ be the $k$-simplex 
$$
\{(t_1, \cdots, t_k) \in \mathbb{R}^k \ | \ 0 \leq t_1 \leq \cdots \leq t_k \leq 1 \}.
$$
We have an evaluation map
$$
\Phi_k : \Delta_k\times LM \rightarrow M^k \\
$$  
defined by
$$
\Phi_k(t_1, \cdots, t_k; \gamma) = (\gamma(t_1), \cdots, \gamma(t_k)).
$$  Then define $P_k$ to be the composition
$$
(\Lambda^*M)^{\otimes k} \rightarrow \Lambda^*M^k \buildrel \Phi_k^* \over \rightarrow \Lambda^*(\Delta_k\times LM) \buildrel p_* \over \rightarrow \Lambda^{*-k}LM
$$
where $p_*$ is the integration along the fiber of the projection $p : \Delta_k\times LM \rightarrow LM $. 

Given $\omega_1$, $\cdots \omega_k$ $\in$ $\Lambda^*M$, the {\it iterated integral}
$$
\int\omega_1\cdots\omega_k
$$
is a differential form on $LM$ of total degree $|\omega_1|+\cdots|\omega_k|-k$, defined by the formula 
$$
\int \omega_1\cdots \omega_k = (-1)^{(k-1)|\omega_1|+(k-2)|\omega_2|+\cdots+|\omega_{k-1}|+k(k-1)/2}P_k(\omega_1,\cdots, \omega_k).
$$ 

\section{Preliminaries}
In this section, we give a construction of some complexes. Let $A$ be an arbitrary differential graded algebra in this section. Let $A^{\vee}$ denote the dual of $A$. The bar complex of $A$, $(B(A), d_B)$, is defined by
$$
B(A) = \oplus_{r \geq 0}\otimes^{r} sA, 
$$
\begin{eqnarray*}
d_{B}(\omega_{1},\cdots ,\omega_{r} ) = -(-1)^{\varepsilon_{i-1}}\sum_i(\omega_1, \cdots, \omega_{i-1},d\omega_i, \omega_{i+1}, \cdots, \omega_r)  \\ 
- (-1)^{\varepsilon_{i}}\sum_i(\omega_1, \cdots, \omega_{i-1}, \omega_{i}\wedge\omega_{i+1},\omega_{i+2}, \cdots, \omega_r).
\end{eqnarray*}
Here $(sA)^q = A^{q+1}$ or $A^q$ according as 0 $\leq$ $q$ or 0 $<$ $q$, and $\varepsilon_i = deg(\omega_1,\cdots,\omega_i)$. We denote the totality of degree $n$ elements by $B(A)_n$. The coproduct $H^*(B(A))$ $\rightarrow$ $H^*(B(A))\otimes H^*(B(A))$ is defined by
$$
(\omega_1, \cdots, \omega_n) \mapsto \sum_i (\omega_1, \cdots, \omega_i) \otimes (\omega_{i+1}, \cdots, \omega_n). 
$$
\\
Chen proved the following theorem.

\begin{theorem}[Chen \cite{chen77.2}] Let $M$ be a simply-connected manifold and $H_*(M)$ be of finite type. Let $A$ be a differential graded algebra of $\Lambda M$ such that $A^0$ = $\mathbb{R}$ and $H^*(A)$ $\cong$ $H^*(\Lambda M)$ by the inclusion. Then there is an isomorphism of coalgebras
$$
 H^*(B(A)) \cong H^*(\Omega M)
$$
given by
$$
(\omega_1, \cdots, \omega_n) \mapsto \int \omega_1 \cdots \omega_n.
$$ 
\end{theorem}
Let $F^pB(A)$ be a filtration of $B(A)$ such that
$$
F^pB(A) = \oplus_{0 \leq r \leq p}\otimes^{r} sA.
$$
Let ${{\rm{Hom}}}(B(A), A^{\vee})_n$ = $\sum_{p+q=n}{{\rm{Hom}}}(B(A)_p, A^{q\vee})$ and 
${{\rm{Hom}}}(B(A), A^{\vee})$ = \\ $\sum_n{{\rm{Hom}}}(B(A), A^{\vee})_n$. Its boundary is defined by
\begin{eqnarray*}
 \lefteqn{\delta\varphi(\omega_{1},\cdots,\omega_{r})(\omega) }\\
 & = & \varphi(\omega_{1},\cdots,\omega_{r}) (d\omega) + (-1)^{|\omega|}\varphi(d_{B}(\omega_{1},\cdots,\omega_{r}))(\omega) \\
&& \mbox{}-(-1)^{|\omega|}\varphi(\omega_{2},\cdots,\omega_{r})(\omega \wedge \omega_{1}) \\
&& +(-1)^{|\omega |+\varepsilon_{r-1}(|\omega_r |+1)}\varphi (\omega_{1},\cdots ,\omega_{r-1})(\omega \wedge \omega_{r}). 
\end{eqnarray*}

Let us define the subcomplex of ${{\rm{Hom}}}(B(A), A^{\vee})$, $\overline{{\rm{Hom}}}(B(A), A^{\vee}$), according to the Chen's normalization of the cyclic bar complex (see \cite{chen77.1} or \cite{gjp}). We define $\overline{{\rm{Hom}}}(B(A), A^{\vee}$) to be the set of elements in ${{\rm{Hom}}}(B(A), A^{\vee})$ which satisfy the following equations for any $\omega, \omega_i$ $\in$ $A^{>0}$ and $f$ $\in$ $A^0$:
\begin{eqnarray*}
\left\{
\begin{array}{l}
-\varphi (\cdots \omega_{i-2}, f\omega_{i-1}, \omega_i, \cdots)( \omega ) + \varphi (\cdots, \omega_{i-1} ,f\omega_{i}, \omega_{i+1}, \cdots)( \omega ) \\
\hspace{3.16cm} + \varphi ( \cdots ,\omega_{i-1} ,df, \omega_{i}, \cdots )(\omega ) = 0, \ \ \mbox{$1\leq i\leq r-1$}, \\
- \varphi (\omega_{1},\cdots ,\omega_{r})(f\omega) + \varphi(f\omega_{1},\cdots ,\omega_{r})(\omega)+  \varphi(df,\omega_{1},\cdots ,\omega_{r})(\omega)=0, \\
- \varphi (\omega_1,\cdots,fw_r)(\omega) + \varphi(\omega_{1},\cdots ,\omega_{r})(f\omega)+ \varphi (\omega_{1},\cdots ,\omega_{r},df)(\omega)=0. 
\end{array}
\right.
\end{eqnarray*}
It can be easily seen that it is isomorphic to the dual of the normalized cyclic bar complex of $A$: 
$$
\overline{\rm{Hom}}(B(A), A^{\vee}) \cong C(A)^{\vee}
$$
Similarly, let ${{\rm{Hom}}}(B(A), A)_n$ = $\sum_{p-q=n}{{\rm{Hom}}}(B(A)_p, A^q)$ and ${{\rm{Hom}}}(B(A), A)$ = $\sum_n{{\rm{Hom}}}(B(A), A)_n$. Its boundary is defined by 
\begin{eqnarray*}
\lefteqn{\delta\varphi(\omega_1, \cdots, \omega_r) }\\
& = & (-1)^{|\varphi|-\varepsilon_r}d\varphi(\omega_1, \cdots, \omega_r) -(-1)^{|\varphi|-\varepsilon_r}\varphi(d_B(\omega_1,\cdots, \omega_r)) \\
&& + (-1)^{|\varphi|-\varepsilon_r}\omega_1\wedge\varphi(\omega_2,\cdots, \omega_r) \\ 
&& -(-1)^{(|\omega_r|+1)(|\varphi|+1)}\varphi(\omega_1\cdots,\omega_{r-1})\wedge\omega_r.
\end{eqnarray*}
We define $\overline{{\rm{Hom}}}(B(A), A)$ to be the set of elements in ${{\rm{Hom}}}(B(A), A)$ which satisfy the following equations for any $\omega, \omega_i$ $\in$ $A^{>0}$ and $f$ $\in$ $A^0$:
\begin{eqnarray*}
\left\{
\begin{array}{l}
-\varphi (\cdots \omega_{i-2}, f\omega_{i-1}, \omega_i, \cdots) + \varphi (\cdots, \omega_{i-1} ,f\omega_{i}, \omega_{i+1}, \cdots) \\
\hspace{2.53cm} + \varphi ( \cdots ,\omega_{i-1} ,df, \omega_{i}, \cdots ) = 0, \ \ \mbox{$1 \leq i \leq r-1$}, \\
- f\wedge\varphi (\omega_{1},\cdots ,\omega_{r})+ \varphi(f\omega_{1},\cdots ,\omega_{r}) +  \varphi(df,\omega_{1},\cdots ,\omega_{r}) =0, \\
- \varphi (\omega_1,\cdots,fw_r) + \varphi(\omega_{1},\cdots ,\omega_{r})\wedge f+ \varphi (\omega_{1},\cdots ,\omega_{r},df)=0. 
\end{array}
\right.
\end{eqnarray*}
The cup product on ${{\rm{Hom}}}(B(A),A)$ is defined by 
\begin{eqnarray*}
\lefteqn{\varphi_1\cup \varphi_2(\omega_1,\cdots, \omega_r) } \\
 & = & \sum_{0\leq i \leq r}(-1)^{|\varphi_1|(|\varphi_2|+\varepsilon_r-\varepsilon_i)}\varphi_1(\omega_1, \cdots, \omega_i)\wedge\varphi_2(\omega_{i+1},\cdots, \omega_r).
\end{eqnarray*}
Since $\delta(\varphi_1\cup\varphi_2) = \delta\varphi_1\cup \varphi_2 + (-1)^{|\varphi_1|}\varphi_1\cup\delta\varphi_2$, $H_*({{\rm{Hom}}}(B(A), A))$ becomes an algebra. This product can be induced on $H_*(\overline{{\rm{Hom}}}(B(A), A))$. 

The $E_1$-term of their spectral sequences associated with the filtration $F^pB(A)$ can be calculated from the cohomology of $A$.

\begin{proposition} There is an isomorphism of vector spaces
$$
H_*(\overline{{\rm{Hom}}}(F^pB(A)/F^{p-1}B(A), A^{\vee})) \cong {{\rm{Hom}}}(\otimes^p sH(A), H(A)^{\vee}) 
$$
\end{proposition}
\begin{proof} 
Let $\overline{A}$ be a differential graded subalgebra of $A$ such that $\overline{A}^p$ = $A^p$ for $p$ $>$ 1, $\overline{A}^0$ = $\mathbb{R}$ and 
$$
A^1 = dA^0 \oplus \overline{A}^1.
$$
There is an isomorphism of vector spaces
$$ 
\overline{{\rm{Hom}}}(F^qB(A)/F^{q-1}B(A), A^{\vee}) \cong {{\rm{Hom}}}(F^qB(\overline{A})/F^{q-1}B(\overline{A}), \overline{A}^{\vee}). \\
$$
Since $\overline{A}^0$ = $\mathbb{R}$, there is an isomorphism
$$
H_0({{\rm{Hom}}}(F^qB(\overline{A})/F^{q-1}B(\overline{A}), \overline{A}^{\vee})) \cong {{\rm{Hom}}}(\otimes sH(\overline{A}), H(\overline{A})^{\vee}).
$$
Therefore we obtain the proposition.
\end{proof}
\section{Proof of Theorem 1.1}
We give the proof of theorem 1.1 in this section. There is a differential graded subalgebra of $A$, $\overline{A}$, such that $\overline{A}^0$ = $\mathbb{R}$ and $H(A)$ $\cong$ $H(\overline{A})$ by the inclusion. Then we obtain the isomorphism of algebras
$$
H_*(\overline{{\rm{Hom}}}(B(A), A)) \cong H_*({{\rm{Hom}}}(B(\overline{A}), \overline{A}))
$$ 
by proposition 3.2. Therefore it suffices to verify the theorem in the case $A^0 = \mathbb{R}$. The following result is due to Chen. 
\begin{theorem}[Chen \cite{chen77.2}] $ H_*(LM) \cong H_*({{\rm{\overline{Hom}}}}(B(A), A^{\vee})). $
\end{theorem}
\begin{proof} 
We define $\psi : C_*(LM) \rightarrow {{\rm{\overline{Hom}}}}(B(A), A^{\vee})$ by
$$
\psi(\sigma)(\omega_1, \cdots, \omega_n)(\omega) = \int_\sigma \pi^*\omega\wedge\int\omega_1\cdots\omega_n.
$$
Let $F_pC_*(LM)$ be a filtration of $C_*(LM)$ such that
\begin{eqnarray*}
F_pC_r(LM) = \mbox{ \{ $\sigma$ : $\Delta^r$ $\rightarrow$ $LM$ $|$ $\pi$ $\circ$ $\sigma$ = $\sigma'$ $\circ$ $\pi'$ for some $\sigma'$ $\in$ $C_q(M)$, } \\
\mbox{ $q$ $\leq$ $p$, $\pi'$ : $\Delta^r$ $\rightarrow$ $\Delta^q$ \} }. 
\end{eqnarray*}
Let $\{ E^r_{p,q} \}$ be the associated spectral sequence. 
Define a filtration of ${{\rm{\overline{Hom}}}}(B(A), A^{\vee})$ by
$$
F_p{{\rm{\overline{Hom}}}}(B(A), A) = \{ f\in {{\rm{\overline{Hom}}}}(B(A), A^{\vee}) \ | \ f(\omega_1, \cdots, \omega_n)(\omega) = 0, \ \forall \omega \in A^{\geq p+1} \}.
$$
It can be easily shown that $\psi$ preserves the filtrations of $C_*(LM)$ and ${{\rm{\overline{Hom}}}}(B(A), A^{\vee})$. On $E_2$-level, the map
$$
\psi : H_p(M) \otimes H_q(\Omega M) \rightarrow H_p(A^{\vee}) \otimes H_q(B(A)^{\vee})
$$
is given by
$$
\sigma_1\otimes\sigma_2 \longmapsto \Bigr(\omega \mapsto \int_{\sigma_1}\omega \Bigl) \otimes \Bigr( (\omega_1, \cdots, \omega_n \mapsto \int_{\sigma_2}\int\omega_1\cdots\omega_n)\Bigl).
$$
Theorem 3.1 asserts that this is an isomorphism. Therefore we obtain the theorem. 
\end{proof}
\begin{lemma} $H_*(\overline{{\rm{Hom}}}(B(A), A)) \cong H_{*-d}(\overline{{\rm{Hom}}}(B(A), A^{\vee})).$ 
\end{lemma}
\begin{proof}
We define a chain map $P$ : $\overline{{\rm{Hom}}}(B(A), A) \rightarrow \overline{{\rm{Hom}}}(B(A), A^{\vee})$ by
$$
P(\varphi)(\omega_1,\cdots, \omega_n)(\omega) = \int_M \omega\wedge\varphi(\omega_1,\cdots, \omega_n).
$$
Define a filtration of $\overline{{\rm{Hom}}}(B(A), A)$ by
$$
F_p\overline{{\rm{Hom}}}(B(A), A) = \{ \varphi \in \overline{{\rm{Hom}}}(B(A), A) \ | \ \varphi(\omega_1, \cdots, \omega_n) \in A^{\geq d-p} \}.
$$
The map $P$ preserves those filtrations. On $E_2$-level, the map
$$
P : H^{d-p}(A) \otimes H_q(B(A)^{\vee}) \rightarrow H_p(A^{\vee}) \otimes H_q(B(A)^{\vee})
$$
is given by
$$
\omega\otimes\varphi \longmapsto \Bigr( \tau \mapsto \int_M \omega\wedge\tau \Bigl) \otimes \varphi.
$$ 
This is isomorphic and we obtain the lemma.
\end{proof}
\begin{proof}[\it Proof of theorem 1.1] \: We can verify that $\mathbb{H}_*(LM)$ is isomorphic to \\
$H_*(\overline{{\rm{Hom}}}(B(A), A))$ as vector spaces by composing the maps in theorem 4.1 and lemma 4.2. We can also verify that there is an isomorphism of associative, commutative algebras. Indeed, the cup product of $\overline{{\rm{Hom}}}(B(A), A)$ on $E_2$-level
$$
H^{d-p}(A) \otimes H_q(B(A)^{\vee}) \otimes H^{d-s}(A) \otimes H_t(B(A)^{\vee}) \rightarrow H^{2d-p-s}(A) \otimes H_{q+t}(B(A)^{\vee})
$$
is given by
$$
a\otimes g \otimes b \otimes h \mapsto (-1)^{(d-p+q)(d-s)} a \wedge b \otimes g\cdot h,
$$
where $g\cdot h$ satisfies
$$
g\cdot h(\omega_1, \cdots, \omega_n) = \sum_i g(\omega_1, \cdots, \omega_i)h(\omega_{i+1}, \cdots, \omega_n). 
$$
Then the following theorem asserts that the loop product and the cup product coincide on $E_2$-level.
\begin{theorem}[Cohen-Jones-Yan \cite{cjy}] Let $M$ be a simply-connected manifold. Then $\{ E^r_{p,q} \}$ becomes an algebra and converges to $H_*(LM)$ as algebras. On $E_2$-level, the product
$$
\mu : H_p(M ; H_q(LM)) \otimes H_s(M ; H_t(LM)) \rightarrow H_{p+q-d}(M ; H_{s+t}(LM))
$$ 
is given by 
$$
\mu((a\otimes g)\otimes(b\otimes h)) = (-1)^{(d-s)(p+q-d)}(a\cdot b)\otimes(gh)
$$
where $a\in H_p(M), b\in H_s(M), g\in H_q(\Omega M), h\in H_t(\Omega M)$, $a\cdot b$ is the intersection product and $gh$ is the Pontryagin product. 
\end{theorem}
Therefore we obtain the theorem.  
\end{proof}

\section{The conjugacy classes of fundamental groups}

Let $\pi$ denote a fundamental group of a smooth manifold $M$ and $J$ denote an augmentation ideal of the group ring of $\pi$, $\mathbb{R}\pi$. Chen showed that the completion of the fundamental group with respect to the powers of its augmentation ideal is isomorphic to the dual of the 0-th cohomology of the bar complex of differential forms via iterated integrals \cite{chen75}:
$$
\varprojlim_p\mathbb{R}\pi/J^p \cong H^0(B(A))^{\vee}
$$
where $A$ is a differential graded subalgebra of $\Lambda M$ such that $A^0 = \mathbb{R}$ and $H^*(A) \cong H^*(M)$. 

Based on this work,  we study iterated integrals on the free loop space of the non-simply-connected manifold. Let $\tilde{\pi}$ denote the set of conjugacy classes of $\pi$ and $\tilde{J^p}$ denote pr($J^p$) where pr is the projection of $\mathbb{R}\pi$ onto $\mathbb{R}\tilde{\pi}$. 

\begin{theorem} 
Let $M$ be a smooth manifold and $H_*(M)$ is of finite type.  Let $A$ be a differential graded subalgebra of $\Lambda M$ such that the map $H^q(A) \rightarrow H^q(\Lambda M)$ induced by the inclusion is isomorphic if $q$ = 0, 1 and injective if $q$ = 2. Then there is an isomorphism of  vector spaces
$$
\varprojlim_p\mathbb{R}\tilde{\pi}/\tilde{J^p} \cong H_0(\overline{{\rm{Hom}}}(B(A), A^{\vee}).
$$
\end{theorem}
We give the proof of this theorem in this section. Let $*$ be a fixed point in $S^1$. In this section,  let $LM$ be a set of smooth maps from  $S^1$ to $M$ which are constant maps near $*$. Let $\Omega_xM$ be a subspace of $LM$ whose elements send $*$ to $x$ $\in$ $M$. Let Diff$(S^1,*)$ denote diffeomorphisms of $S^1$ which coincide with identity map near $*$. We define $\alpha$, $\beta$ : $\Delta^q$ $\rightarrow$ $LM$ to be {\it equivalent by a reparameterization} iff there is a smooth map $\tau$ : $\Delta^q$ $\rightarrow$ Diff($S^1, *$) such that
$$
\beta (\xi)(t) = \alpha (\xi)(\tau(t,\xi)), \ \ \ \ \forall (t, \xi) \in S^1 \times \Delta^q.
$$ 

Let $\overline{C}_*(LM)$ be a chain complex having as a basis the totality of equivalence classes of smooth simplexes of $LM$. Let $\overline{C}_*(\Omega_xM)$ be a chain complex having as a basis the totality of equivalence classes of smooth simplexes of $\Omega_xM$. 
$\overline{C_*}(\Omega_{x}M)$ becomes a noncommutative associative algebra as follows. The product of  $\sigma_1$ and $\sigma_2$ in $\overline{C_*}(\Omega_{x}M)$ is defined to be the path product or 0 according as  deg$\sigma_1$+deg$\sigma_2$ $\leq$ 1 or $>$ 1.  
The augmentation $\varepsilon$ : $\overline{C_*}(\Omega_xM)$ $\rightarrow$ $\mathbb{R}$ is given by $\varepsilon\sigma$ = 1 or 0 according as deg$\sigma$ = 0 or $>$ 0. 

Let $\sigma$ be a smooth simplex of $M$. Define for each $\sigma$
$$
\overline{C_q}(LM)(\sigma)  = \{ \sum n_i\tau_i \in \overline{C_q}(LM) \ | \ \pi_{\sharp}\tau_i=\sigma \}.
$$
$\overline{C_q}(LM)(\sigma)$ becomes a noncommutative associative algebra.  
Let $\varepsilon(\sigma)$ denote the augmentation of $\overline{C_q}(LM)(\sigma)$, given by  $\sum{n_i\tau_i}$ $\mapsto$ $\sum{n_i}$.
Define a filtration of $\overline{C_q}(LM)(\sigma)$ by 
$$ 
 F_p\overline{C_q}(LM) = ({\rm{ker}}\varepsilon)^p \oplus (\oplus_{\sigma : \Delta^q \to M} ({\rm{ker}}\varepsilon(\sigma))^p). 
$$ 
\\

\begin{proposition} The map $\psi_p$ : $F_p\overline{C_q}(LM)$ $\rightarrow$ $\overline{{\rm{Hom}}}(F^{p-1}B(A), A^{\vee})$ given by 
$$
\sigma \mapsto \Bigl( (\omega_1, \cdots, \omega_p) \mapsto \Bigl( \omega \mapsto \int_{\sigma} \pi^* \omega \wedge \int \omega_1 \cdots \omega_p \Bigr) \Bigr)
$$
is well-defined, chain map and $F_p\overline{C_q}(LM) \subset {\rm{ker}}\psi_p$. 
\end{proposition}
\begin{proof}
The well-definedness can be verified by the following lemma which can be verified as in proposition 1.5, proposition 4.1.1 \cite{chen73}, and in proposition 1.5.3 \cite{chen77.2}.
 \begin{lemma}[Chen] 
(1) If $\alpha$ and $\beta$ $\in$ $C_*(LM)$ are equivalent by a reparameterization, then
$$
 \alpha^*\int\omega_1\cdots\omega_n = \beta^*\int\omega_1\cdots\omega_n. 
$$
 (2) If  $\tau_1,\tau_2 \in \overline{C_q}(LM)(\sigma)$, then
$$
\ \ \ (\tau_1\cdot\tau_2)^*\int\omega_1\cdots\omega_n = \sum\tau_1^*\int\omega_1\cdots\omega_i\wedge\tau_2^*\int\omega_{i+1}\cdots\omega_n.
$$
 (3) If f $\in$ $\Lambda^0M$, then for any i
$$
\ \ - \int\omega_1\cdots f\omega_{i-1}\cdots\omega_n + \int\omega_1\cdots f\omega_i\cdots\omega_n + \int\omega_1\cdots\omega_{i-1} \hspace{0.03cm} df \hspace{0.03cm} \omega_i\cdots\omega_n = 0.
$$
\end{lemma}
To verify $F_p\overline{C_q}(LM) \subset {\rm{ker}}\psi_p$, it suffices to show $({\rm{ker}} \varepsilon(\sigma))^p \subset {\rm{ker}} \psi_p$. Let $s$ denote the section of $\pi$, which sends points of $M$ to
the constant map. Take ($\sigma_1-s_\sharp\sigma$) $\cdot$ ($\sigma_2-s_\sharp\sigma$) $\cdot$ $\cdots$ $\cdot$($\sigma_p-s_\sharp\sigma$) $\in $ (${\rm{ker}}\varepsilon (\sigma))^p$, where $\sigma \in C_q(M)$ and $\sigma_i \in \overline{C_q}(LM)(\sigma)$. Then

\begin{eqnarray*}
\lefteqn{ \int_{\Delta^q}  (\sigma_1-s_\sharp\sigma)\cdot(\sigma_2-s\sigma) \cdot\cdots\cdot(\sigma_p-s_\sharp\sigma)^*\Bigr(\pi^*\omega\wedge\int\omega_1\cdots\omega_{p-1}\Bigl) } \\
 & = & \sum_{k = 1}^p \int_{\Delta^q} \sigma^*\omega\wedge(\sigma_1-s_\sharp\sigma)^*\int\omega_1\cdots(\sigma_k-s_\sharp\sigma)^*1\cdots\wedge(\sigma_p-s_\sharp\sigma)^*\int\omega_{p-1} \\
 & = & 0. 
\end{eqnarray*}
Therefore we obtain the proposition.
\end{proof}

Let $C_*(M,x)$ denote a set of smooth simplexes of $M$ neighborhood of whose vertices are at $x$ in $M$. We define 
$$
C\otimes sC^{\otimes p} =  C_*(M,x)\otimes sC_*(M,x)^{\otimes{p}}.
$$
Here $(sC_*(M, x))_q = C_{q+1}(M, x)$ or 0 according as $q > 0$ or $q \leq 0$. Its boundary is given by the sum of the boundary on each complex.
Let us construct a chain map $\Phi$ : $C\otimes sC^{\otimes p}$ $\rightarrow$ $F_p\overline{C_*}(LM)/F_{p+1}\overline{C_*}(LM)$ considering the following three cases:\\
{\bf case 1:}  If $(\sigma_1, \cdots, \sigma_p)$ $\in$ $\Bigr( sC(M,x)^{\otimes p} \Bigl)_1$, then
$$
\Phi : (\sigma_1, \cdots, \sigma_p) \longmapsto  (\sigma_1-x)\cdot(\sigma_2-x) \cdot\cdots\cdot(\sigma_p-x)
$$
where $x$ is regarded as a constant map.\\
{\bf case 2:}  If $(\sigma_1, \cdots, \sigma_p)$ $\in$ $\Bigr( sC(M,x)^{\otimes p} \Bigl)_1$, then
$$
\Phi : (\sigma_1, \cdots, \sigma_p) \longmapsto (\sigma_1-x)\cdot(\sigma_2-x) \cdots\overline{\sigma_i}\cdots(\sigma_p-x) 
$$
where $\overline{\sigma_i}$ : $\Delta^1$ $\ni$ $\xi$ $\mapsto$ $\overline{\sigma_i}(\xi)(t)$ $\in$ $\Omega_xM$ is  

\begin{eqnarray*}
\lefteqn{\overline{\sigma_i} (\xi) (t) } \\
& = & \begin{cases}
\sigma_i((1- \xi )((1-t)v_0+tv_2)+ \xi (1-2t)v_0+2 \xi tv_1), & \mbox{if} \ 0 \leq t \leq 1/2 \\
\sigma_i ((1- \xi )((1-t)v_0+tv_2)+ \xi (2-2t)v_1+ \xi (2t-1)v_2), & \mbox{if} \ 1/2 \leq t \leq 1\\
\end{cases}
\end{eqnarray*}
Here $v_0, v_1, v_2$ are the vertices of the standard simplex $\Delta^2$. \\
{\bf case 3:} If $(\gamma, \sigma_1, \cdots, \sigma_p)$ $\in$ $C_1(M,x) \otimes\Bigr( sC(M,x)^{\otimes p} \Bigl)_0$, then
$$
\Phi : (\gamma, \sigma_1, \cdots, \sigma_p) \longmapsto \gamma_t^{-1}(\sigma_1-x)\gamma_t\cdots\gamma_t^{-1}(\sigma_p-x)\gamma_t
$$
where $\gamma_t : [0, 1] \ni s \mapsto \gamma (st)$ $\in$ $M$, $t$ $\in \Delta^1$. 
\\                                                           

\begin{lemma} The following diagram commutes:
$$\begin{CD}
\Bigr( C\otimes sC^{\otimes p} \Bigl)_1 @ > \Phi >> F_p\overline{C_1}(LM)/F_{p+1}\overline{C_1}(LM)\\
@VV\partial V @VV\partial' V\\
\Bigr( C\otimes sC^{\otimes p} \Bigl)_0 @ > \Phi >> F_p\overline{C_0}(LM)/F_{p+1}\overline{C_0}(LM)\\
\end{CD}$$
\end{lemma}

\begin{proof}
For case 2,
\begin{eqnarray*}
\lefteqn{ \partial'\Phi (\sigma_1, \cdots, \sigma_p) - \Phi\partial (\sigma_1, \cdots, \sigma_p) } \\
 & = & (\sigma_1-x)\cdots(\sigma_i^{(0)}\cdot\sigma_i^{(2)}-\sigma_i^{(1)}-\sigma_i^{(0)}+\sigma_i^{(1)}-\sigma_i^{(2)}+x)\cdots ( \sigma_p -x) \\
 & = & ( \sigma_1 -x) \cdots (\sigma_i^{(0)}-x) \cdot ( \sigma_i^{(2)} -x)\cdots (\sigma_p-x) \in F_{p+1}\overline{C_0}(LM)  
\end{eqnarray*}
 where $\sigma_i^{(1)}$, $\sigma_i^{(2)}$, $\sigma_i^{(3)}$ are the faces of $\sigma_i$. \\
\ \ For case 3, 
\begin{eqnarray*}
\lefteqn{  \partial'\Phi(\gamma, \sigma_1, \cdots, \sigma_p) - \Phi\partial'(\gamma, \sigma_1, \cdots, \sigma_p)} \\
 & = & \gamma^{-1}\cdot(\sigma_1-x)\cdot\gamma\cdots\gamma^{-1}\cdot(\sigma_p-x)\cdot\gamma
-(\sigma_1-x)\cdots(\sigma_p-x) \\
 & \in & \! \! \! \! F_{p+1}\overline{C_0}(LM). 
\end{eqnarray*} 
Therefore we obtain the lemma.
\end{proof} 
Proposition 5.2 gives the map
$$
H_q(F_p\overline{C}(LM)/F_{p-1}\overline{C}(LM)) \rightarrow H_q(\overline{{\rm{Hom}}}(F^pB(A)/F^{p-1}B(A),A^{\vee})). 
$$

\begin{lemma} For $q$ = 0, the following map is isomorphic:
$$
H_0(F_p\overline{C}(LM)/F_{p+1}\overline{C}(LM)) \cong H_0(\overline{{\rm{Hom}}}(F^pB(A)/F^{p-1}B(A),A^{\vee})).
$$  
\end{lemma}
\begin{proof}
We obtain the following surjection by lemma 5.4.
$$
\Phi : H_0(C\otimes sC^{\otimes p}) \twoheadrightarrow H_0(F_p\overline{C}(LM)/F_{p+1}\overline{C}(LM)).
$$
Composing with the isomorphism $\otimes^pH_1(M) \cong H_0(C\otimes sC^{\otimes p})$, the map  
$$
\otimes^pH_1(M) \twoheadrightarrow H_0(F_p\overline{C}(LM)/F_{p+1}\overline{C}(LM)) \rightarrow {\rm{Hom}}(\otimes^pH^1(A), \mathbb{R})\\
$$
is given by
$$
(\sigma_1, \cdots, \sigma_n) \mapsto \Bigr( (\omega_1, \cdots, \omega_p) \mapsto \int_{\sigma_1}\omega_1\cdots\int_{\sigma_p}\omega_p \Bigl).
$$
This is isomorphic and we obtain the lemma. 
\end{proof}

\begin{lemma} For $q$ = 1, the following map surjective:
$$
H_1(F_p\overline{C}(LM)/F_{p+1}\overline{C}(LM)) \twoheadrightarrow H_1({\rm{Hom}}(F^pB(A)/F^{p-1}B(A),A^{\vee})).
$$
\end{lemma}
\begin{proof}
It suffices to show that the following map obtained by lemma 5.4 is surjective.
$$
{\rm{ker}} \partial \rightarrow H_1(F_p\overline{C}(LM)/F_{p+1}\overline{C}(LM)) \rightarrow {\rm{Hom}}(\otimes^p sH(A),H(A)^{\vee})_1 \\
$$
If $(\gamma, \sigma_1, \cdots, \sigma_p)$ $\in$ ${\rm{ker}}\partial$ $\cap$ $\biggr( C_0(M,x)\otimes\Bigr( sC(M,x)^{\otimes p} \Bigl)_1 \biggl)$, then
$$
(\gamma, \sigma_1, \cdots, \sigma_p) \mapsto  \Biggr( (\omega_1, \cdots, \omega_p) \mapsto \biggr( \omega \mapsto
 \begin{cases} \int_\gamma \omega \int_{\sigma_1}\omega_1\cdots\int_{\sigma_p}\omega_p, & \mbox{if} \  deg \ \omega = 0 \\
0, & \mbox{otherwise} \\
\end{cases} 
\biggl) \Biggl) 
$$
through the above map. \\
If $(\gamma, \sigma_1, \cdots, \sigma_p)$ $\in$ ${\rm{ker}}\partial$ $\cap$ $\biggr($ $C_1(M,x) \otimes\Bigr( sC(M,x)^{\otimes p} \Bigl)_0$ $\biggl)$, then
$$
(\gamma, \sigma_1, \cdots, \sigma_p) \mapsto \Biggr( (\omega_1, \cdots, \omega_p) \mapsto \biggr( \omega \mapsto \int_{\gamma}\omega\int_{\sigma_1}\omega_1\cdots\int_{\sigma_p}\omega_p 
\biggl) \Biggl)   
$$
when deg $\omega$ = 1. 
Then we can verify the surjectivity and obtain the lemma. 
\end{proof}

\begin{proof}[\it Proof of theorem 1.1] 
Consider the spectral sequences of $\overline{C}(LM)/F_p\overline{C}(LM)$ and $\overline{{\rm{Hom}}}(F^{p-1}B(A),A^{\vee})$ associated with $F_q\overline{C}(LM)$ and $\overline{{\rm{Hom}}}(F^qB(A),A^{\vee})$, respectively. 
Lemma 5.5 asserts that $\psi_p$ is isomorphic on $E_1$-level at degree $0$:
$$
H_0(F_q\overline{C}(LM)/F_{q+1}\overline{C}(LM)) \cong H_0(\overline{{\rm{Hom}}}(F^{q}B(A)/F^{q-1}B(A),A^{\vee})).
$$ 
Lemma 5.6 asserts that $\psi_p$ is surjective on $E_1$-level at degree $1$: 
$$
H_1(F_q\overline{C}(LM)/F_{q+1}\overline{C}(LM)) \twoheadrightarrow H_1(\overline{{\rm{Hom}}}(F^qB(A)/F^{q-1}B(A),A^{\vee})).
$$
Then there is an isomorphism on $E_r$-level at degree 0 for $r$ $\geq$ $1$.
We have
$$
\mathbb{R}\tilde{\pi}/\tilde{J^p} \cong H_0(\overline{C}(LM)/F_p\overline{C}(LM)) \cong H_0(\overline{{\rm{Hom}}}(F^pB(A), A^{\vee})). 
$$
Therefore we obtain the theorem. 
\end{proof}

\section{The Goldman bracket}

This section is devoted to the proof of the following theorem.
\begin{theorem} 
Let $M$ be a compact closed oriented surface with genus g. Then the Goldman bracket induces a Lie algebra structure on $\varprojlim_p\mathbb{R}\tilde{\pi}/\tilde{J^p}$and there is an isomorphism of Lie algebras 
$$
\varprojlim_p\mathbb{R}\tilde{\pi}/\tilde{J^p} \cong H_0({{\rm{Hom}}}(B(H^*(M)), H^*(M)^{\vee})).
$$
\end{theorem}
Goldman showed that the vector space spanned by the free homotopy classes of closed curves on a closed oriented surface has a Lie algebra structure \cite{goldman}. This work led Chas and Sullivan to the string topology. We would verify that this structure makes $\varprojlim_p\mathbb{R}\tilde{\pi}/\tilde{J^p}$ a Lie algebra. On the other hand, we can construct a bracket on $H_0({{\rm{Hom}}}(B(H^*(M)), H^*(M)^{\vee}))$ by the cup product defined in section 3 and the Connes's operator. Here we regard $H^*(M)$ as a differential graded algebra with a trivial differential. Theorem 6.1 asserts that those two Lie algebras are isomorphic.

First we describe a relation between this bracket and the augmentation ideal of the group ring of the surface group to induce a Lie algebra structure on $\varprojlim_p\mathbb{R}\tilde{\pi}/\tilde{J^p}$. Then we construct a bracket on $H_0(\overline{{\rm{Hom}}}(B(A), A^{\vee}))$ and verify the isomorphism of Lie algebras
$$
\varprojlim_p\mathbb{R}\tilde{\pi}/\tilde{J^p} \cong H_0(\overline{{\rm{Hom}}}(B(A), A^{\vee})).
$$
Finally we verify the isomorphism 
$$
H_0(\overline{{\rm{Hom}}}(B(A), A^{\vee})) \cong H_0({\rm{Hom}}(B(H^*(M)), H^*(M)^{\vee}). 
$$

The following proposition makes $\varprojlim_p\mathbb{R}\tilde{\pi}/\tilde{J^p}$ a Lie algebra.

\begin{proposition}
(1) If $p$ $\geq$ $1$ and $q$ $\geq$ $2$, then $[\tilde{J^p}, \tilde{J^q}] \subset \tilde{J}^{p+q-2}.$ \\
(2) If p $\geq $ $2$ , then $[\tilde{J^p}, \mathbb{R}\tilde{\pi}] \subset \tilde{J}^{p-1}.$ 

\end{proposition}
\begin{proof}
We give a proof of {\it(1)}. Take $(\sigma_1-x)\cdots(\sigma_p-x)$ $\in$ $\tilde{J^p}$, $(\tau_1-y)\cdots(\tau_q-y)$ $\in$ $\tilde{J^q}$, where $\sigma_i$ $\in$ $\Omega_xM$ and $\tau_i$ $\in$ $\Omega_yM$. Assume that all curves are immersions and $\sigma_i$ $\tau_j$ intersect transversally for any $i, j$. Let  $\{\sigma_i\sharp\tau_j\}$ denote the set of intersection points of $\sigma_i$ and $\tau_j$. Also assume that all the intersection points are distinct i.e.  $\{\sigma_i\sharp\tau_j\}$ $\cap$ $\{\sigma_k\sharp\tau_l\}$  = $\phi$ if $i$ $\not=$ $k$ or $j$ $\not=$ l. Then,   
\begin{eqnarray*}
\lefteqn{[\sigma, \tau] = \sum_{i,j}\sum_{s\in\sigma_i\sharp\tau_j} \{ \varepsilon(s ; \sigma_i, \tau_j) \gamma_{s,x}\cdot (\sigma_i-x) \cdots (\sigma_p-x) (\sigma_1-x) \cdots  } \\
 & & \ \ \ \ \ \ \cdot (\sigma_{i-1}-x) \cdot\gamma_{s,x}^{-1}\cdot \cdot \gamma_{s,y}\cdot ( \tau_j-y) \cdots( \tau_q-y )( \tau_1-y) \cdots ( \tau_{j-1}-y)\cdot\gamma_{s,y}^{-1}  \\
 & & \ \ - \gamma_{s,x}\cdot (\sigma_{i+1}-x) \cdots (\sigma_p-x) (\sigma_1-x) \cdots (\sigma_{i-1}-x) \cdot \gamma_{s,x}^{-1}\cdot \\
& & \ \ \ \ \ \ \cdot \gamma_{s,y}\cdot ( \tau_
{j+1}-y) \cdots( \tau_q-y )( \tau_1-y) \cdots ( \tau_{j-1}-y)\cdot\gamma_{s,y}^{-1} \} \\ 
 && \in \tilde{J}^{p+q-2}. 
\end{eqnarray*}
Here $\gamma_{s,x}$ is a path from $s$ to $x$ along $\sigma_i$ and $\gamma_{s,y}$ is a path from $s$ to $y$ along $\tau_j$.  The proof of {\it(2)} can be verified in the same way.
\end{proof}
Let A be a differential graded subalgebra of $\Lambda M$ such that $H^*(A) \cong H^*(\Lambda M)$ by the inclusion. 
\begin{proposition}
There is an isomorphism of vector spaces 
$$
H_*(\overline{{\rm{Hom}}}(F^pB(A), A)) \cong H_{*-2}(\overline{{\rm{Hom}}}(F^pB(A), A^{\vee})).
$$
\end{proposition}
\begin{proof}
We define $P : H_{*-2}(\overline{{\rm{Hom}}}(F^pB(A), A)) \rightarrow H_*(\overline{{\rm{Hom}}}(F^pB(A), A^{\vee}))$ by 
$$
P(\varphi)(\omega_1, \cdots, \omega_p)(\omega) = \int_M \omega\wedge\varphi(\omega_1, \cdots, \omega_p).
$$
This map preserves the filtrations. On $E^1$-level, the map 
$$
{\rm{Hom}}(\otimes^qH(A), H(A)) \rightarrow {\rm{Hom}}(\otimes^qH(A), H(A)^{\vee})
$$
is isomorphic. Therefore we obtain the proposition.
\end{proof}
Now we construct a bracket on $H_0(\overline{{\rm{Hom}}}(B(A), A^{\vee}))$. First, we define  the {\it Connes's operator} $B : H_*(\overline{{\rm{Hom}}}(F^pB(A), A^{\vee})) \rightarrow H_{*+1}(\overline{{\rm{Hom}}}(F^{p-1}B(A), A^{\vee}))$ by
\begin{eqnarray*}
\lefteqn{B(\varphi)(\omega_1, \cdots, \omega_{p-1})(\omega)} \\
& = & \sum_{0\leq k\leq p-1}(-1)^{(\varepsilon_k+1)(\varepsilon_{p-1}-\varepsilon_k)}\varphi(\omega_{k+1}, \cdots ,\omega_{p-1},\omega,\omega_1,\cdots\omega_k)(1).
\end{eqnarray*}
Composing these maps and the cup product, we can define a bracket on \\ $H_0(\overline{{\rm{Hom}}}(F^pB(A), A^{\vee}))$ by
$$
[\varphi_1, \varphi_2] = -P(P^{-1}B\varphi_1\cup P^{-1}B\varphi_2) \in H_0(\overline{{\rm{Hom}}}(F^{p-1}B(A), A^{\vee})).
$$

Take 2$g$ closed 1-forms on $M$, $\alpha_1,\cdots,\alpha_g,\beta_1,\cdots\beta_g$, such that $\int_M \alpha_i\wedge\beta_j = \delta_{ij}$. Let $\{ \overline{E}^r_{p. q} \}$ denote the spectral sequence of $\overline{{\rm{Hom}}}(B(A), A^{\vee})$ associated with $F^pB(A)$. Notice that the cyclic group $\mathbb{Z}/p\mathbb{Z}$ acts on $\overline{E}^1_{p, -p}$ $\cong$ ${\rm{Hom}}(\otimes^pH^1(A), \mathbb{R})$ by
$$
\iota\varphi(\omega_1, \cdots, \omega_p) = \varphi(\omega_2, \cdots, \omega_p, \omega_1)
$$
where $\iota$ is a generator of $\mathbb{Z}/p\mathbb{Z}$. The bracket 
$ [ \ , \ ] : \overline{E}^1_{p, -p} \otimes \overline{E}^1_{q, -q} \rightarrow \overline{E}^1_{p+q-2, -p-q+2} $
is
\begin{eqnarray*}
\lefteqn{[\varphi_1, \varphi_2](\omega_1, \cdots, \omega_{p+q-2})} \\
 & & = \sum_{i, m, n}\iota^m\varphi_1(\alpha_i, \omega_1, \cdots, \omega_{p-1})\varrho^n \varphi_2(\beta_i, \omega_p, \cdots, \omega_{p+q-2}) \\
 & & \hspace{2cm} -\iota^m\varphi_1(\beta_i, \omega_1, \cdots, \omega_{p-1})\varrho^n \varphi_2(\alpha_i, \omega_p, \cdots, \omega_{p+q-2}) \\  
\end{eqnarray*}
where $\iota$ and $\varrho$ are generators of $\mathbb{Z}/p\mathbb{Z}$ and $\mathbb{Z}/q\mathbb{Z}$, respectively. \\

\begin{proposition}
The following diagram commutes for $p, q$ $\geq $ 1:
$$\begin{CD}
\tilde{J^p}/\tilde{J^{p+1}} \otimes \tilde{J^q}/\tilde{J^{q+1}} @ >>> \overline{E}_{\infty }^{p, -p} \otimes \overline{E}_{\infty }^{q, -q} \\
 @ V [ \ , \ ] VV @ V [ \ , \ ] VV \\
\tilde{J}^{p+q-2}/\tilde{J}^{p+q-1} @ >>> \overline{E}_{\infty }^{p+q-2, -p-q+2} 
\end{CD}$$
\end{proposition}
\begin{proof} Take $\sigma$ = $(\sigma_1-x)\cdots(\sigma_p-x)$ $\in$ $F_p\overline{C_0}(LM)$,  $\tau$ = $(\tau_1-y)\cdots(\tau_q-y)$ $\in$ $F_q\overline{C_0}(LM)$. Take 2$g$ curves in $M$, $a_i, b_i$, as in Figure 1. Assume that $\sigma_i$ and $\tau_j$, $a_k$, or $b_k$, intersect transversally for any $i, j, k$. Also assume that $\tau_j$ and $a_k$, or $b_k$, intersect transversally for any $j, k$. 

Assume that all the intersection points are distinct. Then for any $i, j, k$,  we can take each tubular neighborhoods of $a_i$ and $b_i$ so that it does not  include some neighborhoods of intersection points of $\sigma_j$ and $\tau_k$. We fix such neighborhoods  of intersection points and denote them by $U_p$ for each $p$. We can also take a tubular neighborhood of the diagonal map from $M$ to $M$$\times$$M$ outside those neighborhoods of intersection points of $\sigma_i$ and $\tau_j$ for any $i, j$  i.e.
$$
N_{\Delta} \cap \Bigr(\sigma_i \bigr(S^1\setminus\cup_p \sigma_i^{-1}(U_p) \bigl) \times \tau_j \bigr(S^1\setminus\cup_p \tau_j^{-1}(U_p) \bigl) \Bigl) = \phi, \ \forall i, j.
$$ 
Here $N_{\Delta}$ denotes the tubular neighborhood of the diagonal map. Thom class $\Phi$ of this tubular neighborhood satisfies
$$
\int_{\sigma_i|_{\sigma_i^{-1}(U_p)}\times\tau_j|_{\tau_j^{-1}(U_p)}} \Phi = -\varepsilon(p; \sigma_i,\tau_j),
$$
where $\varepsilon(p; \sigma_i,\tau_j)$ is the intersecion number of $\sigma_i$ and $\tau_j$ at $p$.
\begin{figure}
\includegraphics[width=15cm, clip]{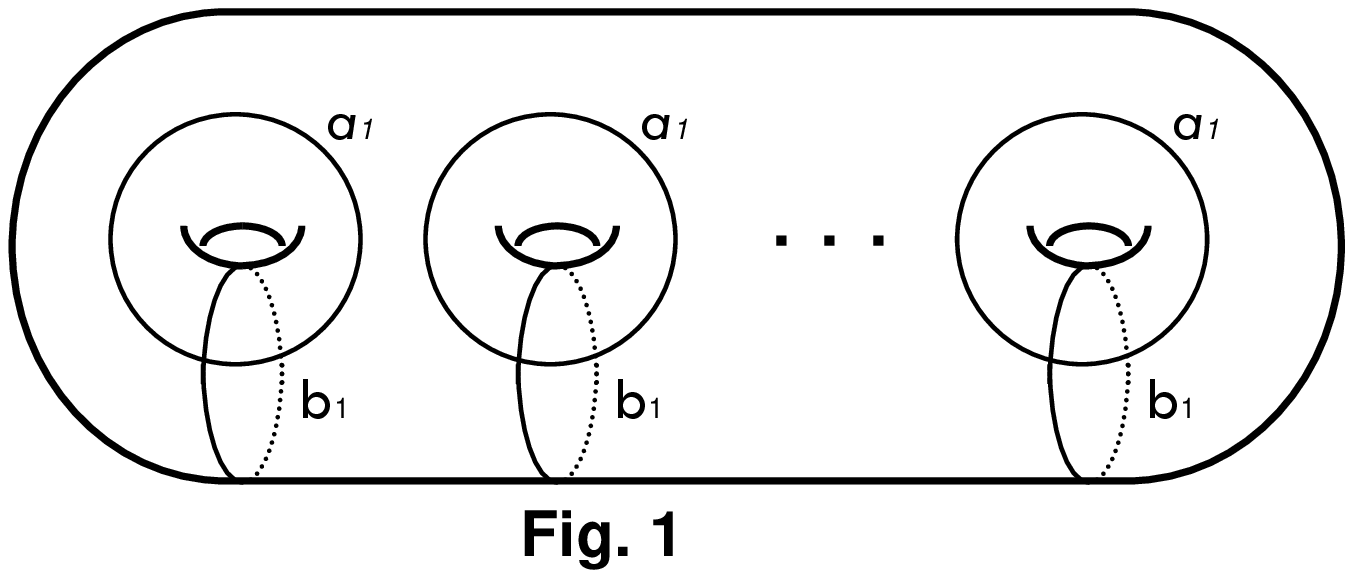}
\end{figure}

Define $e_\sharp$ : $C_0(LM)$ $\rightarrow$ $C_1(LM)$ by
$e_\sharp\gamma(\xi)(t) = \gamma(\xi + t)$.
Let $\omega_k$, $1 \leq k \leq n$,  be differential forms on $M$ which has its support inside the tubular neighborhoods of $a_i$ and $b_i$. Then
\begin{eqnarray*}
\lefteqn{ \int_{[\sigma_i, \tau_j]} \int \omega_1\cdots \omega_{n}  } \\
&& = \sum_{p \in \sigma_i\sharp\tau_j, k}\varepsilon(p ; \sigma_i, \tau_j)\int_{(\sigma_i)_p} \int\omega_1\cdots\omega_k \int_{(\tau_j)_p}\int\omega_{k+1}\cdots\omega_n  \\ 
 & & = - \sum_{p\in\sigma_i\sharp\tau_j, k} \int_{\sigma_i|_{\sigma_i^{-1}(U_p)} \times\tau_j|_{\tau_j^{-1}(U_p)}} \Phi \int_{(\sigma_i)_p}\int\omega_1\cdots\omega_k\int_{(\tau_j)_p} \int\omega_{k+1}\cdots\omega_n \\
 &&  =  - \sum_k \int_{e_\sharp \sigma_i\times e_\sharp \tau_j} \pi^*\Phi \wedge  p_1^*\int \omega_1\cdots \omega_k
\wedge p_2^* \int \omega_{k+1} \cdots \omega_n.
\end{eqnarray*}
Here $p_1, p_2 : LM \times LM \rightarrow LM$ are the projections. The last equality is obtained by the following lemma.

\begin{lemma} If $p$ $\in$ $\sigma_i\sharp\tau_j$ and $p'$ $\in$ $U_p\cap\sigma_i([0, 1])$, then 
$$
\int_{(\sigma_i)_p}\int \omega_1\cdots\omega_n = \int_{(\sigma_i)_{p'}}\int \omega_1\cdots\omega_n.
$$
\end{lemma}
\begin{proof}F
Let $\gamma$ be the curve from $p$ to $p'$ along $\sigma_i$ inside $U_p$. 
If $\gamma$ and $\sigma$ are in the same direction, then
\begin{eqnarray*}
\begin{split}
\int_{(\sigma_i)_{p'}} & \int \omega_1\cdots\omega_n  = \int_{\gamma\cdot(\sigma_i)_{p'}} \int\omega_1\cdots\omega_n  = \int_{(\sigma)_p\cdot\gamma}\int \omega_1\cdots\omega_n   \\ 
& =\int_{(\sigma)_p}\int \omega_1\cdots\omega_n. 
\end{split}
\end{eqnarray*}
We can also verify the case where $\gamma$ is in the direction opposite to $\sigma$ in the same way. 
\end{proof}
We have the equality
\begin{eqnarray*} 
\lefteqn{ \sum_k - \int_{e_\sharp\sigma \times e_\sharp\tau } \pi^* \Phi \wedge p_1^*\int\omega_1\cdots\omega_k \wedge p_2^*\int\omega_{k+1}\cdots\omega_{p+q-2} } \\
 && = \sum_{j, k}\int_{e_\sharp\sigma\times e_\sharp\tau} \pi^*\Bigr( -p_1^*(\alpha_1\wedge\beta_1)-p_2^*(\alpha_1\wedge\beta_1) + p_1^*\alpha_j\wedge p_2^*\beta_j - p_1^*\beta_j\wedge p_2^*\alpha_j \Bigl)  \\ 
 && \hspace{2cm} \wedge p_1^*\int\omega_1\cdots\omega_k\wedge p_2^*\int \omega_{k+1}\cdots\omega_{p+q-2} \\
\end{eqnarray*}
In fact, if $\eta$ $\in$ $\Lambda (M \times M)$ then
\begin{eqnarray*}
\lefteqn{ (-1)^{|\eta|+1}\int_{e_\sharp\sigma \times e_\sharp\tau } \pi^*d\eta \wedge p_1^*\int\omega_1\cdots\omega_k\wedge p_2^*\int\omega_{k+1}\cdots\omega_{p+q-2}} \\
  & & = \int_{e_\sharp\sigma \times e_\sharp\tau } \pi^*\eta\wedge d\Bigr( p_1^*\int\omega_1\cdots\omega_k\wedge p_2^*\int\omega_{k+1}\cdots\omega_{p+q-2} \Bigl) \\
  & & \hspace{1cm}+  (e_\sharp\sigma)^*\int\omega_1\cdots \omega_k \bigwedge  (e_\sharp\tau)^*\int\omega_{k+1}\cdots\omega_j\wedge\omega_{j+1}\cdots\omega_{p+q-2} \Bigl) \\
  & & = 0.
\end{eqnarray*}
The last equality is obtained by the following lemma.
\begin{lemma} If $\sigma \in F_p\overline{C_0}(LM)$, then
$$ 
(e_\sharp\sigma)^*\int\omega_1\cdots\omega_{p-2} = 0.
$$
\end{lemma}
\begin{proof}
It suffices to show the case $\sigma$ = $(\tau_1-x)\cdots(\tau_p-x)$ where $x$ $\in$ M and $\tau_i$ $\in$ $\Omega_xM$. We define $\bar{\tau_i}$ $\in$ $\Omega_xM$ by
$$
\bar{\tau_i}(t) = \begin{cases}
\tau_i(pt), & \mbox{if}\; (i-1)/p\leq t\leq i/p \\
0, & \mbox{otherwise}. \\
\end{cases}  
$$
Let $\bar{\sigma}$ denote $(\bar{\tau_1}-x)\cdots(\bar{\tau_p}-x)$. It can be shown that $e_\sharp\bar{\sigma}$ restricted on $[(i-1)/p , i/p]$ is contained in $F_{p-1}\overline{C_1}(LM)$ for any $i$. Therefore  
$$
(e_\sharp\sigma)^*\int\omega_1\cdots\omega_{p-2} =(e_\sharp\bar{\sigma})^*\int\omega_1\cdots\omega_{p-2} = 0.
$$    
\end{proof}

Jones, Geztler, and Petrack describes the map $e_\sharp$ in terms of iterated integrals by the following theorem.
\begin{theorem}[Geztler-Jones-Petrack \cite{gjp}] If $\sigma$ $\in$ $C_0(LM)$ and $\omega, \omega_i$ $\in$ $\Lambda^1 M$, 1 $\leq$ i $\leq$ p, then
$$
\int_{e_\sharp\sigma} \pi^*\omega\wedge\int\omega_1\cdots\omega_p = \sum_k\int_\sigma\int\omega_k\cdots\omega_p\omega\omega_1\cdots\omega_{k-1}.
$$ 
\end{theorem}
This theorem asserts the equality
\begin{eqnarray*}
\lefteqn{ \sum_{j, k}\int_{e_\sharp\sigma\times e_\sharp\tau} \pi^*\Bigr( -p_1^*(\alpha_1\wedge\beta_1)-p_2^*(\alpha_1\wedge\beta_1) + p_1^*\alpha_j\wedge p_2^*\beta_j - p_1^*\beta_j\wedge p_2^*\alpha_j \Bigl) } \\ 
 && \hspace{2cm} \wedge p_1^*\int\omega_1\cdots\omega_k\wedge p_2^*\int \omega_{k+1}\cdots\omega_n \\
 && = \sum_{j, k, l} \int_{\sigma} \int \omega_{k+1}\cdots\omega_{p-1}\alpha_j\omega_1\cdots\omega_k \int_{\tau} \int \omega_{l+1}\cdots\omega_{p+q-2}\beta_j\omega_p\cdots\omega_l \\
 && - \int_{\sigma} \int \omega_{k+1}\cdots\omega_{p-1}\beta_j\omega_1\cdots\omega_k \int_{\tau} \int \omega_{l+1}\cdots\omega_{p+q-2}\alpha_j\omega_p\cdots\omega_l 
\end{eqnarray*}
Finally we obtain the equality
\begin{eqnarray*}
\lefteqn{ \int_{[\sigma, \tau]} \int \omega_1\cdots \omega_{p+q-2}  } \\
 && = \sum_{j, k, l} \int_{\sigma} \int \omega_{k+1}\cdots\omega_{p-1}\alpha_j\omega_1\cdots\omega_k \int_{\tau} \int \omega_{l+1}\cdots\omega_{p+q-2}\beta_j\omega_p\cdots\omega_l \\
 && - \int_{\sigma} \int \omega_{k+1}\cdots\omega_{p-1}\beta_j\omega_1\cdots\omega_k \int_{\tau} \int \omega_{l+1}\cdots\omega_{p+q-2}\alpha_j\omega_p\cdots\omega_l
\end{eqnarray*}
Since we can take $\omega_i$ $\in$ $H^1(M)$, $1 \leq i \leq p+q-2$, so that their support are inside the tubular neighborhoods of $a_j$ and $b_j$,  we obtain the proposition.
\end{proof}

\begin{proof}[\it Proof of theorem 6.1.]
We obtain the following isomorphism of Lie algebras by proposition 6.4.
$$
\varprojlim_p\mathbb{R}\tilde{\pi}/\tilde{J^p} \cong H_0(\overline{{\rm{Hom}}}(B(A), A^{\vee}).
$$
To obtain the isomorphism of Lie algebras
$$
H_0(\overline{{\rm{Hom}}}(B(A), A^{\vee}) \cong H_0({\rm{Hom}}(B(H^*(M)), H^*(M)^{\vee}),
$$
we introduce the following lemma, which asserts the formality of the compact  ${\rm K\ddot{a}hler}$ manifolds.
\begin{lemma}[$dd^c Lemma$, Deligne-Griffiths-Morgan-Sullivan \cite{dgms}] Let X be a compact $K\ddot{a}hler$ manifold and $d^c = J^{-1}dJ$ where J gives the complex structure in the cotangent bundle. If $\alpha$ is a differential form on X such that d$\alpha$ = 0 and $d^c\alpha$ = 0, and such that $\alpha = d\gamma$, then $\alpha = dd^c\beta$ for some $\beta$.
\end{lemma}
{\bf Cor.} {\it There are quasi-isomorphisms of differential graded algebras }
$$
(\Lambda X, d) \leftarrow ({\rm{ker}}d^c, d) \rightarrow (H_{d^c}^*(X), 0).
$$
Notice that a closed oriented surface endowed with a complex structure become a $\rm{K\ddot{a}hler}$ manifolds for the dimensional reason. Therefore the following lemma completes the proof of the theorem. 
\begin{lemma} If $f : A_1 \rightarrow A_2$ is a quasi-isomorphism of differential graded algebras, then the map induced by $f$
$$
H_0(\overline{{\rm{Hom}}}(B(A_1), A_1^{\vee}) \rightarrow H_0(\overline{{\rm{Hom}}}(B(A_2), A_2^{\vee})
$$
is an isomorphism. 
\end{lemma}
\begin{proof}
It suffices to verify that the map induced by $f$
$$
\overline{f} : H_0(\overline{{\rm{Hom}}}(F^pB(A_1), A_1^{\vee}) \rightarrow H_0(\overline{{\rm{Hom}}}(F^pB(A_2), A_2^{\vee})
$$
is an isomorphism for any $p$. On $E^1$-level, the map induced by $f$
$$
{\rm Hom}(\otimes sH(A_1), H(A_1)^{\vee}) \rightarrow {\rm Hom}(\otimes sH(A_2), H(A_2)^{\vee}))
$$
is an isomorphism because $f$ is quasi-isomorphism. Therefore we obtain the lemma.
\end{proof}
Therefore we obtain the theorem.
\end{proof}

\end{document}